\documentclass[a4paper, 12pt, reqno]{amsart}

\usepackage{amsmath}
\usepackage{amsthm}
\usepackage{amssymb}
\usepackage{setspace}
\usepackage{enumerate}
\usepackage{comment}
\usepackage{bm}
\usepackage[initials, nobysame]{amsrefs}
\renewcommand{\MR}[1]{}
\makeatletter
\DeclareRobustCommand{\mrnumber}[1]{}
\makeatother

\newtheorem{thm}{Theorem}[section]
\newtheorem{lem}[thm]{Lemma}

\theoremstyle{definition}

\newtheorem{rem}[thm]{Remark}

\numberwithin{equation}{section}

\newcommand{\pt}{\partial}

\newcommand{\R}{\mathbb{R}}

\renewcommand{\epsilon}{\varepsilon}

\DeclareMathOperator{\re}{Re}
\DeclareMathOperator{\im}{Im}

\title[Scattering for fourth-order NLS]
{Scattering of non-radial solutions for the fourth-order Schr\"{o}dinger equation with $\dot{H}^{1/2}$-critical and supercritical nonlinearity}

\author{Takahisa Inui}
\address{Department of Mathematics, Graduate School of Science, The University of Osaka, Japan 560-0043.}
\email{inui@math.sci.osaka-u.ac.jp}

\author{Koichi Komad}
\address{Research Institute for Science and Engineering, Waseda University, Japan 169-8555}
\email{k-komada@aoni.waseda.jp}

\author{Kuranosuke Nishimura}
\address{ARISE Analytics Inc., THE LINKPILLAR 1 NORTH 10F, 2-21-1 Takanawa, Minato-ku, Tokyo, Japan 108-8618}
\email{1117614@alumni.tus.ac.jp}



\date{}

\begin{document}

\maketitle

\noindent
{\bf Abstract.} We consider the focusing nonlinear fourth-order Schr\"{o}dinger equation
\begin{equation*}
    i \pt_ t u + \mu \Delta u - \Delta^2 u 
    = - |u|^{p-1} u, \ \ \ (t,x) \in \R \times \R^d,
\end{equation*}
where $\mu \ge 0$.
The scattering of radial solutions below the ground state is shown by \cite{Guo, Dinh} in the case where $d \ge 2$ and $1 + 8/d < p < 1 + 8/(d-4)_{+}$.
In the present paper, we prove the scattering of non-radial solutions below the ground state in the case where $d \ge 2$ and $1 + 8/(d-1) \le p < 1 + 8/(d-4)_{+}$.
Our proof is based on the concentration compactness argument by Kenig--Merle \cite{KM06}. To overcome difficulties in the non-radial case, we use the virial identity in the direction orthogonal to the momentum vector.
The condition $p \ge 1 + 8/(d-1)$ is required to show the positivity of the functional in that virial identity.

\section{Introduction} \label{sec:intro}

We consider the focusing nonlinear fourth-order Schr\"{o}dinger equation
\begin{equation} \label{eq:4NLS}
      \left\{
      \begin{split}
          & i \pt_t u + \mu \Delta u - \Delta^2 u
          = - |u|^{p-1} u, \ \ \ (t,x)\in \R \times \R^d, \\
          & u |_{t=0} = u_0 \in H^2 (\R^d),
      \end{split}
      \right.
\end{equation}
where $u=u(t,x)$ is a complex-valued unknown function and $\mu \ge 0$ is a constant.
The classical mathematical model for propagation of intense laser beams in a bulk medium with Kerr nonlinearity is given by the nonlinear Schr\"{o}dinger equation (NLS).
To take into account the role of small fourth-order dispersion, Karpman \cite{Karpman}, Karpman--Shagalov \cite{KS} and Fibich--Ilan--Papanicolaou \cite{FIP} studied the nonlinear fourth-order Schr\"{o}dinger equation \eqref{eq:4NLS}.
In Fukumoto--Moffat \cite{FM}, similar equations also appear in the study of a vortex filament in an incompressible fluid.

It is known that \eqref{eq:4NLS} is locally well-posed in $H^2(\R^d)$ and has three conserved quantities; the mass $M(u)$, the energy $E(u)$ and the momentum $P(u)$:
\begin{align*}
      M(u(t))
      := & \ \int_{\R^d} |u(t,x)|^2 dx
      = M(u_0), \\
      E(u(t))
      := & \ \int_{\R^d} \left( 
      \frac{\mu}{2} |\nabla u(t,x)|^2
      + \frac12 |\Delta u(t,x)|^2 
      - \frac{1}{p+1} |u(t,x)|^{p+1}
      \right) dx
      = E(u_0), \\
      P(u(t))
      := & \ \int_{\R^d} \im \left[
      \overline{u}(t,x)\nabla u(t,x) 
      \right] dx
      = P(u_0).
\end{align*}
If $\mu=0$, \eqref{eq:4NLS} is invariant under the scaling transformation 
\begin{align*}
      u(t,x) \mapsto 
      \lambda^{\frac{4}{p-1}} u(\lambda^4 t, \lambda x), 
      \ \ \ \lambda > 0.
\end{align*}
From this scaling, the critical exponent of homogeneous Sobolev spaces is given by
\begin{align*}
      s_c := \frac{d}{2} - \frac{4}{p-1}.
\end{align*}
We say that \eqref{eq:4NLS} is $L^2$-supercritical and $\dot{H}^2$-subcritical if $0 < s_c < 2$. 
This condition is equivalent to $p > 1$ satisfying
\begin{align} \label{eq:intercri}
    1 + \frac{8}{d} < p < 
    \begin{cases}
        \infty & \mbox{if} \ 1 \le d \le 4, \\
        1 + \frac{8}{d-4} & \mbox{if} \ d \ge 5.
    \end{cases}
\end{align}

In this article, we study the global dynamics for \eqref{eq:4NLS} with $p$ satisfying \eqref{eq:intercri}. 
For the defocusing nonlinear fourth-order Schr\"{o}dinger equation
\begin{equation*} 
    i \pt_t u + \mu \Delta u - \Delta^2 u
    = |u|^{p-1} u,
\end{equation*}
Pausader \cite{P07} proved global well-posedness for all radial initial data in $H^2(\R^d)$ when $d \ge 5$ and $1 < p \le 1 + 8/(d-4)$. 
Moreover, he obtained scattering in the $\dot{H}^2$-critical case: $p=1+8/(d-4)$.
Note that the radial assumption is not necessary when $p < 1+8/(d-4)$ and is removed later in \cite{MXZ11, P09JFA} in the remaining $\dot{H}^2$-critical case when $d \ge 8$.
We refer the reader to \cite{PX} for the low dimensional and $L^2$-supercritical cases: $1 \le d \le 4$ and $p > 1+8/d$. 
See also \cite{PS} where global well-posedness and scattering in $L^2$ was proved in the high dimensional $L^2$-critical case: $d\geq5$ and $p=1+8/d$.

Let us turn to the focusing equation \eqref{eq:4NLS}. 
It is known that ground state solutions play an important role in describing global dynamics of nonlinear dispersive equations. 
This is also the case for \eqref{eq:4NLS}. 
Miao--Xu--Zhao \cite{MXZ09} and Pausader \cite{P09DCDS} studied \eqref{eq:4NLS} with $\mu = 0$ in the $\dot{H}^2$-critical case, i.e., the case where $d\ge5$ and $p=1+8/(d-4)$. 
They proved global well-posedness and scattering for radial initial data $u_0$ satisfying $E(u_0) < E(W)$ and $\|u_0\|_{\dot{H}^2(\R^d)} \le \|W\|_{\dot{H}^2(\R^d)}$, where $W$ is a least $\dot{H}^2$ solution to the biharmonic elliptic equation $ \Delta^2 W = |W|^{8/(d-4)} W $.
In \cite{MXZ09,P09DCDS}, the proofs are based on the concentration compactness argument in the spirit of Kenig--Merle \cite{KM06}.
Guo \cite{Guo} studied the scattering problem in the inter-critical case \eqref{eq:intercri}.
When $d \ge 2$ and $\mu = 0$, he proved scattering for radial initial data $u_0$ satisfying
    \begin{equation} \label{eq:scale_prod}
        \left\{
        \begin{split}
            M(u_0)^{\frac{2-s_c}{s_c}} E(u_0) < & \ M(Q)^{\frac{2-s_c}{s_c}} E(Q), \\
            \|u_0\|_{L^2(\R^d)}^{\frac{2-s_c}{s_c}} \|\Delta u_0\|_{L^2(\R^d)} \le & \ \|Q\|_{L^2(\R^d)}^{\frac{2-s_c}{s_c}} \|\Delta Q\|_{L^2(\R^d)},
        \end{split}
        \right.
    \end{equation}
where $Q$ is a ground state solution to $Q + \Delta^2 Q = |Q|^{p-1} Q$.
The proof in \cite{Guo} is based on the concentration compactness argument. 
Dinh \cite{Dinh} provides an alternative proof, avoiding the use of the concentration compactness argument. 
The proof in \cite{Dinh} relies on the approach introduced by Dodson--Murphy \cite{DM}. 
Note that in both \cite{Guo} and \cite{Dinh}, solutions are assumed to be radially symmetric and the one-dimensional case is excluded. 
Komada \cite{Komada} relaxed the assumption of radial symmetry to the assumption of group-invariance.
In particular, the scattering of even solutions was proved in the one dimensional case.
Moreover, in \cite{Komada}, the scattering of solutions below a non-optimal threshold was proved without any symmetry assumptions.
That threshold is the action of the ground state divided by $4^{2/(p-1)}$.

In this paper, we study the scattering of solutions below the ground state without radial symmetry.
For $\omega > 0$, \eqref{eq:4NLS} has standing wave solutions $u(t,x) = e^{i \omega t} \phi(x)$, where $\phi$ is a solution of the fourth-order elliptic equation
\begin{align} \label{eq:SW}
      - \omega \phi + \mu \Delta \phi - \Delta^2 \phi
      = - |\phi|^{p-1} \phi,\ \ \ x\in\R^d.
\end{align}
It is known that for all $ \omega > 0 $, there exists a ground state solution $ Q_{\omega} \in H^2(\R^d) $ to \eqref{eq:SW} such that
\begin{align*} 
    S_{\omega}({Q_{\omega}}) = 
    \inf \{ S_{\omega}(f) \mid f \in H^2(\R^d) \setminus \{0\}, 
    f \ \mbox{is a solution to \eqref{eq:SW}} \},
\end{align*}
where $ S_{\omega}(f) $ is the action defined by
\begin{align*}
    S_{\omega} (f) := \frac{\omega}{2} M(f) + E(f). 
\end{align*}
When $ p $ satisfies \eqref{eq:intercri}, the ground state $ Q_{\omega} $ is characterized by the following alternative minimization problem:
\begin{align} \label{eq:alter_vc}
    S_{\omega}(Q_{\omega}) = 
    \inf \{ S_{\omega}(f) \mid f \in H^2(\R^d) \setminus \{0\},
    K(f) = 0 \},
\end{align}
where $ K(f) $ is the virial functional defined by
\begin{align*}
    K(f) 
    := & \ \left. 
    \frac{\pt}{\pt \lambda} S_{\omega}(T^{\lambda} f)
    \right|_{\lambda = 1} \\
    = & \ \int_{\R^d} \left(
    \mu |\nabla f(x)|^2
    + 2 |\Delta f(x)|^2 
    - \frac{d(p-1)}{2(p+1)} |f(x)|^{p+1}
    \right) dx,
\end{align*}
where $ T^{\lambda}f(x) := \lambda^{\frac{d}{2}} f (\lambda x) $ for $\lambda > 0$. 
We refer to \cite{BCGJ, BCSN, BN, FLNZ} regarding the existence and properties of ground states.

Our main result in this paper is stated as follows.
\begin{thm} \label{thm:main}
    Let $ d \ge 2 $ and assume that $ p $ satisfies
    \begin{align} \label{eq:cond_p}
        1 + \frac{8}{d-1} \le p <
        \begin{cases} 
        \infty & \mbox{if} \ 2 \le d \le 4, \\
        1 + \frac{8}{d-4} & \mbox{if} \ d \ge 5.
        \end{cases}
    \end{align}
    If $ u_0 \in H^2(\R^d) $ satisfies
    \begin{align} \label{eq:below_gs}
        S_{\omega} (u_0) < S_{\omega} (Q_{\omega}) 
        \ \mbox{for some} \ \omega > 0, \ \mbox{and} \ 
        K(u_0) \ge 0,
    \end{align}
    then there exists a unique global solution $ u \in C(\R, H^2(\R^d)) $ to \eqref{eq:4NLS} and $ u $ scatters forward and backward in time, i.e., there exist $ \phi^{\pm} \in H^2(\R^d) $ such that 
    \begin{align} \label{eq:scattering}
        \lim_{t \to \pm \infty}
        \| u(t) - e^{i t (\mu \Delta - \Delta^2)} \phi^{\pm} \|_{H^2(\R^d)} = 0.
    \end{align}
\end{thm}

\begin{rem}
    When $d \ge 2$, the condition that $p \ge 1 + 8/(d-1)$ is equivalent to $s_c \ge 1/2$. We need this condition to prove the positivity of a functional in the virial identity in the direction orthogonal to the momentum vector (see Lemma \ref{lem:key} for details).
    Therefore, the scattering of non-radial solutions below the ground state remains an open problem when $d = 1$ or $1 + 8/d < p < 1 + 8/(d-1)$.
\end{rem}

\begin{rem}
    If $\mu=0$, then \eqref{eq:below_gs} is equivalent to the condition \eqref{eq:scale_prod}.
\end{rem}

\begin{rem}
    When $\mu = 0$, $d \ge 2$, $p$ satisfies \eqref{eq:intercri} and $p \le 9$, Boulenger--Lenzmann \cite{BL} proved that if $u_0 \in H^2(\R^d)$ is radial and satisfies $S_{\omega}(u_0) < S_{\omega}(Q_{\omega})$ for some $\omega > 0$ and $K(u_0) < 0$, then the corresponding solution $u$ to \eqref{eq:4NLS} blows up in finite time.
\end{rem}

The proof of Theorem \ref{thm:main} is based on the concentration compactness argument introduced by Kenig--Merle \cite{KM06}, where they proved global well-posedness and scattering of radial solutions below the ground state for the focusing energy-critical nonlinear Schr\"{o}dinger (NLS) equation. 
We also refer to Holmer--Roudenko \cite{HR} for the focusing 3D cubic NLS in the radial case. 
Roughly speaking, the argument begins by contradiction and  reformulating the problem as a variational problem. 
First, we suppose that the threshold of action for scattering is strictly below the ground state. 
Then, we can construct a critical non-scattering solution standing exactly at the threshold. 
If the problem is restricted to the radial case, we can show that the critical solution is spatially localized near the origin uniformly in time. 
This uniform localization enables us to establish the local virial identity, which leads to a contradiction. 
In \cite{Guo}, the proof follows this strategy. 

We recall that the scattering results for NLS are extended to the non-radial case (see \cite{AN,Dodson15,Dodson19,DHR,FXC,KV10}). 
In the non-radial case, the critical solution is localized near some function $x(t)$. 
For NLS, by using Galilean invariance, one can control the behavior of $x(t)$, which enables the contradiction argument via the local virial identity. 
However, the fourth-order Schr\"{o}dinger equation \eqref{eq:4NLS} is not Galilean invariant. 
Due to this lack of Galilean invariance, the local virial argument does not work for \eqref{eq:4NLS} in the non-radial case.

Let us explain the idea of the proof of Theorem \ref{thm:main}.
Suppose that there exists a non-scattering global solution $u$ to \eqref{eq:4NLS} with initial data $u_0$ satisfying \eqref{eq:below_gs}. 
Then from the variational characterization of the ground state \eqref{eq:alter_vc}, we see that $\inf_{t \in \R} K(u(t)) > 0$.
If $u$ is localized near $x(t)$, then the localized version of the following virial identity centered at $x(t)$ 
\begin{align} \label{eq:original_virial}
    \frac{d}{dt} \int_{\R^d} (x-x(t)) \cdot \im [ \overline{u}(t,x) \nabla u(t,x) ] dx
    = 2 K(u(t)) - \dot{x}(t) \cdot P(u)
\end{align}
implies that
\begin{align} \label{eq:average_K}
    \lim_{T \to \infty} \left( \frac{1}{T} \int_0^T 2 K(u(t)) dt - \frac{x(T) - x(0)}{T} \cdot P(u) \right) = 0.
\end{align}
In the radial case, we can assume that $P(u) = 0$ and $x(t)$ is bounded. 
In this case, \eqref{eq:average_K} contradicts $\inf_{t \in \R} K(u(t)) >0$. 
However, this argument is not applicable in the non-radial case.
In the proof of Theorem \ref{thm:main}, instead of \eqref{eq:original_virial}, we use the virial identity in the direction orthogonal to the momentum vector, which follows the idea of Pausader \cite{P10}.
By rotational symmetry and momentum conservation of \eqref{eq:4NLS}, we may assume that $P(u)$ is always parallel to the first vector $(1,0,...,0) \in \R^d$ without loss of generality.
In this case, if $u$ is localized near $x(t)$, then the localized version of the following identity holds:
\begin{align} \label{eq:virial_perp}
    \frac{d}{dt} \int_{\R^d} (x - x(t))_{\perp} \cdot \im [\overline{u}(t,x) \nabla_{\perp} u(t,x)] dx
    = 2 K_{\perp}(u(t)),
\end{align}
where $x_{\perp} = (x_2,...,x_d) \in \R^{d-1}$, $\nabla_{\perp} = (\frac{\pt}{\pt x_2},...,\frac{\pt}{\pt x_d})$ and
\begin{align*} 
    K_{\perp}(f)
    := \int_{\R^d} \left(
    \mu |\nabla_{\perp} f|^2 
    + 2 |\nabla_{\perp} \pt_1 f|^2 + 2 |\Delta_{\perp} f|^2
    - \frac{(d-1)(p-1)}{2(p+1)} |f|^{p+1}
    \right) dx,
\end{align*}
where $\pt_1 = \frac{\pt}{\pt x_1}$, $\Delta_{\perp} = \sum_{j=2}^{d} \frac{\pt^2}{\pt x_j^2}$.
In \cite{P10}, Pausader used a virial identity similar to \eqref{eq:virial_perp} to prove the large data scattering of non-radial solutions for the defocusing beam equation. 
We also refer to \cite{PS} where the defocusing nonlinear fourth-order Schr\"{o}dinger equation was considered.
In the defocusing case, the functional in the virial identity is always positive.
However, in the focusing case, we need to show the positivity of the functional from its variational properties.
The new key ingredient is the positivity of $K_{\perp}$ below the ground state (see Lemma \ref{lem:key}).
When $p$ satisfies \eqref{eq:cond_p}, we can prove that if $S_{\omega}(f) < S_{\omega}(Q_{\omega})$ for some $\omega > 0$ and $K(f) > 0$, then $K_{\perp}(f) > 0$.
Using this fact and \eqref{eq:virial_perp}, we can derive a contradiction and complete the proof.

We remark that our method is applicable to other nonlinear dispersive equations without Galilean invariance.
In \cite{IN}, the first and last authors prove the scattering of non-radial solutions below the ground state for the quadratic NLS system in $\R^5$ without mass resonance condition.

The rest of the paper is organized as follows: 
In Section \ref{sec:va}, we discuss the properties of ground states and functionals.
In Section \ref{sec:cc}, we define the threshold of action for scattering and apply the concentration compactness argument to construct a critical non-scattering solution. 
Finally, in Section \ref{sec:rigidity}, we complete the proof of Theorem \ref{thm:main} using the localized virial identity in the direction orthogonal to the momentum vector.

\section{Variational analysis} \label{sec:va}

In this section, we discuss the properties of the ground state $Q_{\omega}$ and functionals $S_{\omega}$, $K$, $K_{\perp}$.

\begin{lem} \label{lem:sign_K}
    Let $ d \ge 1 $ and $p$ satisfy \eqref{eq:intercri}.
    Let $ u \in C(I,H^2(\R^d)) $ be a solution to \eqref{eq:4NLS}, where $ I \subset \R $ is a maximal interval of existence. 
    Suppose that $ S_{\omega}(u(t_0)) < S_{\omega}(Q_{\omega}) $ for some $ \omega > 0 $ and $ K(u(t_0)) \ge 0 $ for some $ t_0 \in I $.
    Then $ S_{\omega}(u(t)) < S_{\omega}(Q_{\omega}) $ and $ K(u(t)) \ge 0 $ for all $ t \in I $. 
    In particular, $I = \R$ and $\|u\|_{L_t^{\infty} H_x^2 (\R \times \R^d)} < \infty$.
\end{lem}

\begin{proof}
    By the conservation laws, we see that $S_{\omega}(u(t)) < S_{\omega}(Q_{\omega})$ for all $t \in I$.
    To argue by contradiction, we suppose that there exists $t_1 \in I \setminus \{t_0\}$ such that $K(u(t_1)) < 0$.
    Then, by continuity, there exists $t_2 \in I$ such that $K(u(t_2)) = 0$.
    From \eqref{eq:alter_vc} we see that $u(t_2) = 0$ or $S_{\omega}(u(t_2)) \ge S_{\omega}(Q_{\omega})$.
    These are contradictions.
    Therefore, $K(u(t)) \ge 0$ for all $t \in I$.
    Then, we have
    \begin{align*}
        S_{\omega}(Q_{\omega})
        > & \ S_{\omega}(u) - \frac{2}{d(p-1)} K(u(t)) \\
        = & \ \frac{\omega}{2} \|u(t)\|_{L^2}^2
        + \frac{d + 2 s_c}{4d} \mu \|\nabla u(t)\|_{L^2}^2
        +\frac{s_c}{d} \|\Delta u(t)\|_{L^2}^2
    \end{align*}
    for all $t \in I$, which implies that there exists a constant $C > 0$ independent of $I$ such that $\|u\|_{L_t^{\infty} H_x^2 (I \times \R^d)} < C$.
    Using this uniform bound, the solution $u$ can be extended globally in time, that is, $I = \R$.
\end{proof}

When $d \ge 2$, we decompose $x \in \R^d$ as $x = (x_1,x_{\perp})$ with $x_{\perp} = (x_2,...,x_d)$.
For any $\lambda > 0$, we define the scale transformation $T_{\perp}^{\lambda}$ as
\begin{align*}
    T_{\perp}^{\lambda} f (x)
    := \lambda^{\frac{d-1}{2}} f (x_1, \lambda x_{\perp}).
\end{align*}
Note that $\|T_{\perp}^{\lambda}f\|_{L^2(\R^d)} = \|f\|_{L^2(\R^d)}$ and $\|\pt_1 T_{\perp}^{\lambda} f\|_{L^2(\R^d)} = \|\pt_1 f\|_{L^2(\R^d)}$ for all $\lambda >0$.
Let us define $K_{\perp}$ as
\begin{align*}
    K_{\perp}(f) 
    := & \ \left. \frac{\pt}{\pt \lambda} S_{\omega}(T_{\perp}^{\lambda} f) \right|_{\lambda = 1} \\
    = & \ \int_{\R^d} \left(
    \mu |\nabla_{\perp} f|^2
    + 2 |\nabla_{\perp} \pt_1 f|^2 
    + 2 |\Delta_{\perp} f|^2 
    - \frac{(d-1)(p-1)}{2(p+1)} |f|^{p+1} 
    \right) dx \\
    = & \ K(f) - \int_{\R^d} \left( 
    \mu |\pt_1 f|^2 + 2 |\nabla \pt_1 f|^2 -\frac{p-1}{2(p+1)} \int_{\R^d} |f|^{p+1}  
    \right) dx,
\end{align*}
where $\pt_1 = \frac{\pt}{\pt x_1}$, $\nabla_{\perp} = (\frac{\pt}{\pt x_2},...,\frac{\pt}{\pt x_d})$, $\Delta_{\perp} = \sum_{j=2}^{d} \frac{\pt^2}{\pt x_j^2}$.

\begin{lem} \label{lem:key}
    Let $ d \ge 2 $ and $p$ satisfy \eqref{eq:cond_p}. 
    If $ f \in H^2(\R^d) $ satisfies $ S_{\omega}(f) < S_{\omega}(Q_{\omega}) $ for some $ \omega > 0 $ and $ K(f) > 0 $, then $ K_{\perp}(f) > 0 $.
\end{lem}

\begin{rem}
    When $d \ge 2$ and $p$ satisfies \eqref{eq:cond_p}, for all $\omega > 0$, the ground state $Q_{\omega}$ is characterized as 
    \begin{align*}
        S_{\omega}(Q_{\omega})
        = \inf \{ S_{\omega}(f) \mid f \in H^2(\R^d) \setminus \{0\}, K_{\perp}(f) = 0 \}.
    \end{align*}
    See Lemma \ref{lem:ap} for the details.
\end{rem}

\begin{proof}[Proof of Lemma \ref{lem:key}.]
    Suppose $ K_{\perp}(f) \le 0 $. 
    Since $p \ge 1 + 8/(d-1)$, we have
    \begin{align*}
        \lambda \frac{\pt}{\pt \lambda} S_{\omega}(T_{\perp}^{\lambda} f)
        = & \ \mu \|\nabla_{\perp} f\|_{L^2(\R^d)}^2 \lambda^2
        + 2 \|\nabla_{\perp} \pt_1 f\|_{L^2(\R^d)}^2 \lambda^2
        + 2 \|\Delta_{\perp} f\|_{L^2(\R^d)}^2 \lambda^4 \\
        & \ - \frac{(d-1)(p-1)}{2(p+1)}\|f\|_{L^{p+1}(\R^d)}^{p+1} \lambda^{\frac{(d-1)(p-1)}{2}} \\
        \le & \ K_{\perp}(f) \lambda^4
        \le 0
    \end{align*}
    for all $\lambda \ge 1$.
    This implies that $ \lambda \mapsto S_{\omega}(T_{\perp}^{\lambda} f) $ is non-increasing for all $ \lambda \ge 1 $.
    Since $ p \ge 1 + 8/(d-1) $, we have
    \begin{align} 
        K(T_{\perp}^{\lambda} f)
        = & \ \mu \|\pt_1 f\|_{L^2(\R^d)}^2
        + 2 \|\pt_1^2 f\|_{L^2(\R^d)}^2 
        + \left(
        \mu \|\nabla_{\perp} f\|_{L^2(\R^d)}^2 
        + 4 \|\nabla_{\perp} \pt_1 f\|_{L^2(\R^d)}^2 
        \right) \lambda^2 \notag \\
        & \ + 2 \|\Delta_{\perp} f\|_{L^2(\R^d)}^2 \lambda^4 
        - \frac{d(p-1)}{2(p+1)} \|f\|_{L^{p+1}(\R^d)}^{p+1} \lambda^{\frac{(d-1)(p-1)}{2}} \notag \\
        \le & \ \mu \|\pt_1 f\|_{L^2(\R^d)}^2
        + 2 \|\pt_1^2 f\|_{L^2(\R^d)}^2 
        + 2 \|\nabla_{\perp} \pt_1 f\|_{L^2(\R^d)}^2 \lambda^2 \notag \\ 
        & \ + \left(
        K_{\perp}(f) - \frac{p-1}{2(p+1)} \|f\|_{L^{p+1}(\R^d)}^{p+1}
        \right) \lambda^{\frac{(d-1)(p-1)}{2}}
        \label{eq:K(perp)}
    \end{align}
    for all $\lambda \ge 1$. 
    Hence, using $ K_{\perp} (f) \le 0 $ we see that $ K(T_{\perp}^{\lambda} f) \to - \infty $ as $ \lambda \to \infty $.
    On the other hand, $ K(T_{\perp}^1 f) = K(f) > 0 $.
    Therefore, there exists $ \lambda_0 \in (1,\infty) $ such that $ K(T_{\perp}^{\lambda_0} f) = 0 $, which implies that
    \begin{align*}
        S_{\omega}(Q_{\omega}) 
        \le S_{\omega}(T_{\perp}^{\lambda_0} f)
        \le S_{\omega}(T_{\perp}^1 f)
        = S_{\omega}(f).
    \end{align*}
    This contradicts $ S_{\omega}(f) < S_{\omega}(Q_{\omega}) $.
\end{proof}

\section{Concentration compactness} \label{sec:cc}

If the initial data $u_0 \in H^2(\R^d)$ satisfies \eqref{eq:below_gs}, from \eqref{lem:sign_K} the corresponding solution $u$ globally exists and $\|u\|_{L_t^{\infty} H_x^2 (\R \times \R^d)} < \infty$.
Hence, it suffices to prove that $u$ scatters.
Let us define
\begin{align*}
    \|u\|_{X(I)}
    := \|u\|_{L_t^{q_1} L_x^{r} (I \times \R^d)},
\end{align*}
where $r=p+1$ and $q_1 = \frac{4(p-1)(p+1)}{8-(d-4)(p-1)}$.
Applying the Strichartz type estimates for the fourth-order Schr\"{o}dinger equation, $\|u\|_{X(\R)} < \infty$ implies \eqref{eq:scattering} (see \cite{Komada, KoMaNL}). 

Fix $ \omega > 0 $. For $ L \ge 0 $, we define
\begin{align*}
    C_{\omega}(L)
    := \sup \{ \|u\|_{X(I)}^{q_1} \mid 
    S_{\omega}(u_0) \le L, \ K(u_0) \ge 0 \},
\end{align*}
where the supremum is taken over all solutions $u$ to \eqref{eq:4NLS} defined on $ I \times \R^d $ with the initial data $ u_0 \in H^2(\R^d) $.
Set
\begin{align*}
    L_{\omega}^{*} 
    := \sup \{ L \in [0,S_{\omega}(Q_{\omega})) \mid 
    C_{\omega}(L) < \infty \}.
\end{align*}
Then, to prove Theorem \ref{thm:main}, it suffices to show $ L_{\omega}^{*} \ge S_{\omega}(Q_{\omega}) $ for all $\omega > 0$.
Note that the small data scattering theory shows $L_{\omega}^{*} > 0$.
If $L_{\omega}^{*} < S_{\omega}(Q_{\omega})$, then by the concentration compactness argument, we can construct a critical non-scattering solution with initial data satisfying \eqref{eq:below_gs}.
Moreover, this solution has a precompact orbit in $H^2(\R^d)$ modulo spatial translation.

\begin{lem} \label{lem:critical_sol}
    Let $ d \ge 1 $ and $p$ satisfy \eqref{eq:intercri}.
    Suppose $ L_{\omega}^{*} < S_{\omega}(Q_{\omega}) $.
    Then there exists a global solution $ u \in C(\R,H^2(\R^d)) $ to \eqref{eq:4NLS} such that $ S_{\omega}(u) = L_{\omega}^{*} $, $ K(u(0)) \ge 0 $ and  
    \begin{align} \label{eq:non-scattering}
        \|u\|_{X((-\infty, 0])}
        = \|u\|_{X([0, \infty))}
        = \infty.
    \end{align}
    Moreover, there exists a function $ x(\cdot) \in C^1(\R,\R^d) $ such that $ |\dot{x}(t)| \lesssim_{u} 1 $ for all $ t \in \R $ and $ \{ u(t,\cdot + x(t)) \mid t \in \R \} $ is  precompact in $ H^2(\R^d) $.
\end{lem}

\begin{proof}
    Applying the concentration compactness argument, we can prove that there exists a global solution $u$ to \eqref{eq:4NLS} such that $S_{\omega}(u) = L_{\omega}^{*}$, $K(u(0)) \ge 0$ and \eqref{eq:non-scattering} holds.
    Moreover, there exists a function $x(\cdot) \in C(\R,\R^d)$ such that $\{ u(t,\cdot+x(t)) \mid t \in \R \}$ is precompact in  $H^2(\R^d)$.
    See \cite{Komada, KoMaNL} for the details.
    
    Therefore, it suffices to prove that $x(\cdot) \in C^1(\R,\R^d)$ and $|\dot{x}(t)| \lesssim_u 1$ for all $t \in \R$.
    To prove this, we first claim that there exists $s_0>0$ such that 
    \begin{align*}
        |x(t + s) - x(t)| \lesssim_u 1
    \end{align*}
    for all $t \in \R$ and for all $|s| < s_0$.
    Indeed, if this were not true, then there exist $\{t_n\}$ and $\{s_n\}$ such that $s_n \to 0$ and
    \begin{align} \label{eq:if_not}
        |x(t_n + s_n) - x(t_n)| \to \infty. 
    \end{align}
    Let $v_n$ be a solution to \eqref{eq:4NLS} with initial data $v_n(0) = u(t_n,\cdot + x(t_n))$. 
    Then, by the precompactness, $v_n(0)$ converges to some $v_0$ in $H^2(\R^d)$.
    Let $v$ be a solution to \eqref{eq:4NLS} with initial data $v(0)=v_0$.
    By the stability (Proposition 2.7 in \cite{Komada}), for any compact interval $I \subset \R$, $v_n$ converges to $v$ uniformly on $I$.
    By the continuity of $v$, for all $\epsilon > 0$, there exists $\nu > 0$ such that 
    \begin{align*}
        \|v(s_n) - v_0\|_{H^2(\R^d)} < \epsilon
    \end{align*}
    if $|s_n| < \nu$.
    Hence, for all $\epsilon > 0$, we have
    \begin{align}
        \|v_n(s_n) - v_n(0)\|_{H^2(\R^d)}
      \le & \ \|v_n(s_n) - v(s_n)\|_{H^2(\R^d)}
      + \|v(s_n) - v_0\|_{H^2(\R^d)} \notag \\
      & \ + \|v_0 - v_n(0)\|_{H^2(\R^d)} \notag \\
      < & \ 3 \epsilon \label{eq:3epsilon}
    \end{align}
    for sufficiently large $n$.
    By the uniqueness of \eqref{eq:4NLS}, we have $v_n(s_n) = u(s_n + t_n, \cdot + x(t_n))$.
    Thus, by the precompactness of $\{ u(t,\cdot + x(t)) \mid t \in \R \}$ and \eqref{eq:if_not}, we see that $v_n(s_n)$ converges to $0$ weakly in $H^2(\R^d)$.
    Then, from \eqref{eq:3epsilon}, we see that $v_n(0)$ also converges to $0$ weakly in $H^2(\R^d)$.
    Therefore, $v_0 = 0$ and so $S_{\omega}(u) = \lim_{n \to \infty} S_{\omega}(v_n(0)) = S_{\omega}(v_0) = 0$, which contradicts the fact that $S_{\omega}(u) = L_{\omega}^{*} > 0$.
    
    Let $t_j := j s_0$ and define $\tilde{x}(t)$ by
    \begin{align*}
        \tilde{x}(t) 
        := x(t_j) + \frac12 \left( 1 - \cos \left( \frac{\pi(t-t_j)}{s_0} \right) \right) (x(t_{j+1}) - x(t_j))
    \end{align*}
    for $t_j \le t \le t_{j+1}$.
    Then, $\tilde{x}(t_j) = x(t_j)$ for all $j$.
    Moreover, $\tilde{x} \in C^1(\R)$ and 
    \begin{align*}
        |\dot{\tilde{x}}(t)| \lesssim_u s_0^{-1}
    \end{align*}
    for all $t \in \R$.
    Since $|x(t) - \tilde{x}(t)| \lesssim_u 1$ for all $t \in \R$, we see that $\{ u(t,\cdot + \tilde{x}(t)) \mid t \in \R \}$ is also precompact in $H^2(\R^d)$.
    Thus, replacing $x(t)$ by $\tilde{x}(t)$, we get $|\dot{x}(t)| \lesssim_u 1$.
\end{proof}

\section{Rigidity} \label{sec:rigidity}

In this section, we complete the proof of Theorem \ref{thm:main} by contradiction and the localized virial identity.
To construct the localized virial identity, we use the following lemma.

\begin{lem} \label{lem:tightness}
    Let $ u \in C(\R,H^2(\R^d)) $ be such that $ \{ u (t, \cdot + x(t)) \mid t \in \R \} $ is precompact in $ H^2(\R^d) $ for some function $ x(t) $. 
    Then for all $ \eta > 0 $, there exists $ C(\eta) > 0 $ such that
    \begin{align*}
        \int_{|x-x(t)| \ge C(\eta)} \left(
        |u(t,x)|^2
        + |\nabla u(t,x)|^2 + |\Delta u(t,x)|^2 + |u(t,x)|^{p+1}
        \right) dx
        \le \eta
    \end{align*}
    for all $ t \in \R $.
\end{lem}

\begin{proof}
    See \cite{Komada}.
\end{proof}

Using the localized virial identity in the direction orthogonal to the momentum vector and Lemma \ref{lem:key}, we show the following rigidity theorem.

\begin{thm} \label{thm:rigidity}
    Let $d \ge 2$ and $p$ satisfy \eqref{eq:cond_p}.
    Let $u \in C(\R, H^2(\R^d))$ be a global solution to \eqref{eq:4NLS} with initial data $u_0$ satisfying \eqref{eq:below_gs}.
    If there exists $x(\cdot) \in C^1(\R,\R^d)$ such that $|\dot{x}(t)| \lesssim_u 1$ for all $t \in \R$ and $\{ u(t,\cdot + x(t)) \mid t \in \R \}$ is precompact in $H^2(\R^d)$, then $u \equiv 0$.
\end{thm}

\begin{proof}
    Since $u_0$ satisfies \eqref{eq:below_gs}, from Lemma \ref{lem:sign_K} we see that $S_{\omega}(u) < S_{\omega}(Q_{\omega})$ for some $\omega > 0$ and $K(u(t)) \ge 0$ for all $t \in \R$.
    Moreover, $\|u\|_{L_t^{\infty} H_x^2(\R \times \R^d)} < \infty$.
    From \eqref{eq:alter_vc} and Lemma \ref{lem:key}, we have $K_{\perp}(u(t)) \ge 0$ for all $t \in \R$.
    
    By rotational symmetry and momentum conservation of \eqref{eq:4NLS}, without loss of generality, we may assume that the momentum vector $ P(u) $ is always parallel to the first vector $ e_1 = (1,0,...,0) \in \R^d $, namely, 
    \begin{align*} 
        P_j(u(t)) = \int_{\R^d} \im [\overline{u}(t,x) \pt_j u(t,x)] dx
        = 0, \ \ \ j=2,...,d
    \end{align*}
    for all $t \in \R$. 

    Let $ \psi \in C_0^{\infty}([0,\infty)) $ be a smooth, non-increasing function such that $ \psi(s) = 1 $ for $ 0 \le s \le 1 $ and $ \psi(s) = 0 $ for $ s \ge 2 $.
    For each $ R > 0 $, we define
    \begin{align*}
        \Psi_R (x) :=
        R^2 \int_0^{\frac{|x|}{R}} \int_0^r \psi (s) ds dr.
    \end{align*}
    Then for all $ m \ge 1 $, we have
    \begin{align} \label{eq:2-m}
        \| \nabla^{m} \Psi_R \|_{L^{\infty}}
        \lesssim R^{2-m}.
    \end{align}
    For $1 \le j,k \le d$, the second derivatives of $\Psi_R$ are
    \begin{align} 
        & \pt_j \pt_k \Psi_R (x) \notag \\
        = & \ \frac{R}{|x|} \left( \delta_{jk} - \frac{x_j x_k}{|x|^2} \right) \int_0^{\frac{|x|}{R}} \psi (s) ds
        + \frac{x_j x_k}{|x|^2} \psi \left( \frac{|x|}{R} \right) \notag \\
        = & \ \delta_{jk} 
        + \delta_{jk} \frac{R}{|x|} \int_0^{\frac{|x|}{R}} \left( \psi (s) - 1 \right) ds
        + \frac{x_j x_k}{|x|^2} \frac{R}{|x|} \int_0^{\frac{|x|}{R}} \left( \psi \left( \frac{|x|}{R} \right) -  \psi (s) \right) ds. \label{eq:Psi_R}
    \end{align}
    Since $ \psi (s) = 1 $ for $ 0 \le s \le 1 $, we see that
    \begin{align} \label{eq:<R}
        \int_0^{\frac{|x|}{R}} \left( \psi (s) - 1 \right) ds
        =  
        \int_0^{\frac{|x|}{R}} \left( \psi \left( \frac{|x|}{R} \right) - \psi (s) \right) ds
        = 0 \ \ \ \mbox{if} \ |x| \le R.
    \end{align}
    Put $ z = x - x(t) $ and let us define an operator $Z_j$ by
    \begin{align*}
        Z_j f := & \ \pt_j \Psi_R (z) \pt_j f
        + \pt_j ( \pt_j \Psi_R (z) f ) \\
        = & \ 2 \pt_j \Psi_R (z) \pt_j f
        + \pt_j^2 \Psi_R (z) f.
    \end{align*}
    For $ R > 0 $, we define
    \begin{align*}
        V_{j,R} (t)
        := & \ \frac12 \langle i u(t) | Z_j u(t) \rangle \\
        = & \ \int_{\R^d} \pt_j \Psi_R (z)
        \cdot \im [\overline{u} (t,x) \pt_j u(t,x)] dx,
    \end{align*}
    where $ \langle f | g \rangle := \re \int_{\R^d} f(x) \overline{g}(x) dx $.
    Then
    \begin{align*}
        |V_{j,R}(t)|
        \lesssim & \ R \|u(t)\|_{L^2(\R^d)} \|\nabla u(t)\|_{L^2(\R^d)} \\
        \lesssim & \ R \|u(t)\|_{L^2(\R^d)}^{\frac32} \|\Delta u(t)\|_{L^2(\R^d)}^{\frac12}
    \end{align*}
    for all $ t \in \R $.
    Since $ Z_j^{*} = - Z_j $, we have
    \begin{align} \label{eq:d/dt}
        \frac{d}{d t} V_R^j (t)
        = \langle i u_t | Z_j u \rangle
        + \int_{\R^d} 
        (- \dot{x}_k (t)) \pt_k \pt_j \Psi_R (z)
        \im [\overline{u} \pt_j u] dx.
    \end{align}
    Since $ u $ is a solution to \eqref{eq:4NLS}, using $ Z_j^{*} = - Z_j $, we have
    \begin{align*}
        \langle i u_t | Z_j u \rangle
        = & \ \langle - \mu \Delta u + \Delta^2 u - |u|^{p-1} u | Z_j u \rangle \notag \\
        = & \ - \frac12 \mu \langle [\Delta, Z_j] u | u \rangle
        + \frac12 \langle [\Delta^2, Z_j] u | u \rangle
        - \frac12 \langle [|u|^{p-1}, Z_j] u | u \rangle, 
    \end{align*}
    where $ [A, B] = A B - B A $. From the definition of $ Z_j $, we have 
    \begin{align*}
        [\Delta, Z_j]
        = 2 \pt_k ( \pt_k \pt_j \Psi_R (z) ) \pt_j
        + 2 \pt_j (\pt_j \pt_k \Psi_R (z) ) \pt_k
        + \pt_j^2 \Delta \Psi_R (z).
    \end{align*}
    Using the identity $ \Delta A + A \Delta = 2 \pt_l A \pt_l + [\pt_l, [\pt_l, A]] $ for an operator $ A $, we have
    \begin{align*}
        [\Delta^2, Z_j]
        = & \ \Delta [\Delta, Z_j] + [\Delta, Z_j] \Delta \\
        = & \ 2 \pt_l [\Delta, Z_j] \pt_l 
        +[\pt_l, [\pt_l, [\Delta, Z_j]]] \notag \\
        = & \ 4 \pt_l \pt_k (\pt_k \pt_j \Psi_R (z)) \pt_j \pt_l
        + 4 \pt_l \pt_j (\pt_j \pt_k \Psi_R (z)) \pt_k \pt_l 
        + 2 \pt_l (\pt_j^2 \Delta \Psi_R (z)) \pt_l \notag \\
        & \ + 2 \pt_k (\pt_k \pt_j \Delta \Psi_R (z)) \pt_j
        + 2 \pt_j (\pt_j \pt_k \Delta \Psi_R (z)) \pt_k
        + \pt_j^2 \Delta^2 \Psi_R (z). \notag
    \end{align*}
    Using the above commutator formulas, we integrate by parts to obtain
    \begin{align}
        & \ - \frac12 \mu \langle [\Delta, Z_j] u | u \rangle
        + \frac12 \langle [\Delta^2, Z_j] u | u \rangle \notag \\
        = & \ 2 \mu \int_{\R^d} \pt_j \pt_k \Psi_R(z) \re [\pt_j \overline{u} \pt_k u] dx
        + 4 \int_{\R^d} \pt_j \pt_k \Psi_R (z) \re [\pt_l \pt_j \overline{u} \pt_k \pt_l u] dx 
        \label {eq:L1} \\
        & \ - \frac12 \mu \int_{\R^d} \pt_j^2 \Delta \Psi_R(z) |u|^2 dx
        - \int_{\R^d} \pt_j^2 \Delta \Psi_R (z) |\nabla u|^2 dx \label{eq:L2} \\
        & \ - 2 \int_{\R^d} \pt_j \pt_k \Delta \Psi_R (z) \re [\pt_j \overline{u} \pt_k u] dx  
        + \frac12 \int_{\R^d} \pt_j^2 \Delta^2 \Psi_R (z) |u|^2 dx.
        \label{eq:L3}
    \end{align}
    Using \eqref{eq:Psi_R} and \eqref{eq:<R}, we have
    \begin{align*}
        \eqref{eq:L1}
        = & \ 2 \mu \int_{\R^d} |\pt_j u|^2 dx
        + 4 \int_{\R^d} |\nabla \pt_j u|^2 dx
        + A_R^1(t),
    \end{align*}
    where
    \begin{align*}
        |A_R^1(t)| \lesssim 
        \int_{|x-x(t)| \ge R} \left(
        \mu |\nabla u(t,x)|^2 + |\Delta u(t,x)|^2
        \right) dx.
    \end{align*}
    Using \eqref{eq:2-m}, we have
    \begin{align*}
        | \eqref{eq:L2} + \eqref{eq:L3} |
        \lesssim R^{-2} \|\nabla u\|_{L_t^{\infty} L_x^2 (\R \times \R^d)}^2
        + \left( \mu R^{-2} + R^{-4} \right) M(u).
    \end{align*}
    Using $ [|u|^{p-1}, Z_j] = - 2 \pt_j \Psi_R (z) \pt_j (|u|^{p-1}) $, we integrate by parts to obtain
    \begin{align} 
        - \frac12 \langle [|u|^{p-1}, Z_j] u | u \rangle
        = & \ \int_{\R^d} |u|^2 \pt_j \Psi_R (z) \pt_j (|u|^{p-1}) dx \notag \\
        = & \ \frac{p-1}{p+1} \int_{\R^d} \pt_j \Psi_R (z) \pt_j (|u|^{p+1}) dx \notag \\
        = & \ - \frac{p-1}{p+1} \int_{\R^d} \pt_j^2 \Psi_R (z) |u|^{p+1} dx. \label{eq:N}
    \end{align}
    Using \eqref{eq:Psi_R} and \eqref{eq:<R}, we have
    \begin{align*}
        \eqref{eq:N}
        = - \frac{p-1}{p+1} \int_{\R^d} |u|^{p+1} dx
        +A_R^2(t),
    \end{align*}
    where
    \begin{align*}
        |A_R^2(t)| \lesssim
        \int_{|x-x(t)| \ge R} |u(t,x)|^{p+1} dx.
    \end{align*}
    Therefore, applying Lemma \ref{lem:tightness} we obtain
    \begin{align} \label{eq:virial_j}
        \langle i u_t | Z_j u \rangle
        = \int_{\R^d} \left( 
        2 \mu |\pt_j u|^2
        + 4 |\nabla \pt_j u|^2 
        - \frac{p-1}{p+1} |u|^{p+1} 
        \right) dx
        + o_R(1),
    \end{align}
    where $o_R(1) \to 0$ as $R \to \infty$ uniformly in $t \in \R$.
    Using \eqref{eq:Psi_R}, we have
    \begin{align*}
        - \dot{x}_k(t) \int_{\R^d} \pt_k \pt_j \Psi_R (z) \im [\overline{u} \pt_j u] dx
        = - \dot{x}_j(t) P_j(u) + A_R^3(t),
    \end{align*}
    where
    \begin{align*}
        |A_R^3(t)| \lesssim
        |\dot{x}(t)| 
        \|\nabla u\|_{L_t^{\infty} L_x^2 (\R \times \R^d)} 
        \left( \int_{|x-x(t)| \ge R} |u(t,x)|^2 dx \right)^{\frac12}.
    \end{align*}
    Applying Lemma \ref{lem:tightness} we have
    \begin{align} \label{eq:center}
        - \dot{x}_k(t) \int_{\R^d} \pt_k \pt_j \Psi_R (z) \im [\overline{u} \pt_j u] dx
        = - \dot{x}_j(t) P_j(u) + o_R(1),
    \end{align}
    where $o_R(1) \to 0$ as $R \to \infty$ uniformly in $t \in \R$.
    Therefore, from \eqref{eq:d/dt}, \eqref{eq:virial_j} and \eqref{eq:center} we obtain
    \begin{align} \label{eq:virial_id_j}
        \frac{d}{d t} V_{j,R}(t)
        = \int_{\R^d} \left(
        2 \mu |\pt_j u|^2 + 4 |\nabla \pt_j u|^2 
        - \frac{p-1}{p+1} |u|^{p+1} 
        \right) dx
        - \dot{x}_j(t) P_j(u) + o_R(1),
    \end{align}
    where $o_R(1) \to 0$ as $R \to \infty$ uniformly in $t \in \R$.
    
    Let us define
    \begin{align*}
        V_{\perp, R}(t)
        := \sum_{j=2}^{d} V_{j, R}(t).
    \end{align*}
    Then
    \begin{align*}
        \sup_{t \in \R} |V_{\perp,R}(t)|
        \lesssim R M(u)^{\frac34} \|u\|_{L_t^{\infty} \dot{H}_x^2 (\R \times \R^d)}^{\frac12}.
    \end{align*}
    Since $P_j(u) = 0$ for $j=2,...,d$, from \eqref{eq:virial_id_j}, we obtain
    \begin{align*} 
        \frac{d}{d t} V_{\perp, R} (t)
        = 2 K_{\perp} (u(t)) + o_R(1).
    \end{align*}
    Integrating this on $[0,T]$, we have
    \begin{align} \label{eq:time_average}
        \frac{1}{T} \int_0^T K_{\perp}(u(t)) dt
        \lesssim_u \frac{R}{T} + o_R(1).
    \end{align}
    Taking $ R = T^{\frac12} $ and letting $ T \to \infty $, we see that the right-hand side of \eqref{eq:time_average} can be arbitrarily small. 
    This implies that we can take $ t_n \to \infty $ such that $ K_{\perp}(u(t_n)) \to 0 $.
    By the precompactness of $ \{ u(t, \cdot + x(t)) \mid t \in \R \} $, there exists $ \phi \in H^2(\R^d) $ such that $ u(t_n, \cdot + x(t_n)) $ converges to $ \phi $ in $ H^2(\R^d) $, passing to a subsequence if necessary. 
    By the strong convergence and conservation laws, we see that $S_{\omega}(\phi) = S_{\omega}(u) < S_{\omega}(Q_{\omega})$, $K(\phi) \ge 0$ and $K_{\perp}(\phi) = 0$. 
    Hence, by Lemma \ref{lem:key}, we have $ K(\phi) = 0 $, which implies $\phi = 0$ from \eqref{eq:alter_vc}. 
    Therefore, $M(u) = M(\phi) = 0$ and so $u \equiv 0$.
\end{proof}

\begin{proof}[Proof of Theorem \ref{thm:main}]
    It suffices to show that $L_{\omega}^{*} \ge S_{\omega}(Q_{\omega})$ holds for all $\omega > 0$.
    Fix $\omega > 0$ and suppose $L_{\omega}^{*} < S_{\omega}(Q_{\omega})$ to argue by contradiction.
    Then, from Lemma \ref{lem:critical_sol} there exists a critical solution $u$ to \eqref{eq:4NLS} such that $S_{\omega}(u) = L_{\omega}^{*}$ and $K(u(0)) \ge 0$.
    Moreover, there exists $x(\cdot) \in C^1(\R,\R^d)$ such that $|\dot{x}(t)| \lesssim 1$ for all $t \in \R$ and $\{ u(t,\cdot + x(t)) \mid t \in \R \}$ is precompact in $H^2(\R^d)$.
    Applying Theorem \ref{thm:rigidity}, we see that $u \equiv 0$.
    This contradicts $S_{\omega}(u) = L_{\omega}^{*} > 0$.
\end{proof}

\appendix 
\renewcommand{\theequation}{A.\arabic{equation}} 
\setcounter{equation}{0} 
\section{Characterization of ground state by $K_{\perp}$} \label{sec:ap}

\begin{lem} \label{lem:ap}
    Let $d \ge 2$ and $p$ satisfy \eqref{eq:cond_p}.
    For all $\omega > 0$, $S_{\omega}(Q_{\omega}) = m_{\omega, \perp}$ where
    \begin{align*}
        m_{\omega, \perp}
        := \inf \{ S_{\omega}(f) \mid f \in H^2(\R^d) \setminus \{0\}, K_{\perp}(f) = 0 \}.
    \end{align*}
\end{lem}

\begin{proof}
    Integrating \eqref{eq:SW} multiplied by $\left( x_{\perp} \cdot \nabla_{\perp} + \frac{d-1}{2} \right) \phi$, we have $K_{\perp}(\phi) = 0$.
    Since $Q_{\omega}$ solves \eqref{eq:SW}, $K_{\perp}(Q_{\omega}) = 0$ and so $m_{\omega, \perp} \le S_{\omega}(Q_{\omega})$ for all $\omega > 0$.
    Therefore, it suffices to show $m_{\omega, \perp} \ge S_{\omega}(Q_{\omega})$.
    Suppose that $f \in H^2(\R^d) \setminus \{0\}$ satisfies $K_{\perp}(f) = 0$.
    Then, since $p \ge 1 + 8/(d-1)$, we have
    \begin{align*}
        \lambda \frac{\pt}{\pt \lambda} S_{\omega}(T_{\perp}^{\lambda} f)
        = & \ \mu \|\nabla_{\perp} f\|_{L^2(\R^d)}^2 \lambda^2
        + 2 \|\nabla_{\perp} \pt_1 f\|_{L^2(\R^d)}^2 \lambda^2
        + 2 \|\Delta_{\perp} f\|_{L^2(\R^d)}^2 \lambda^4 \\
        & \ - \frac{(d-1)(p-1)}{2(p+1)} \|f\|_{L^{p+1}(\R^d)}^{p+1} \lambda^{\frac{(d-1)(p-1)}{2}} \\
        \le & \ K_{\perp}(f) \lambda^4 = 0
    \end{align*}
    for all $\lambda \ge 1$.
    On the other hand, 
    \begin{align*}
        \lambda \frac{\pt}{\pt \lambda} S_{\omega}(T_{\perp}^{\lambda} f)
        \ge K_{\perp}(f) \lambda^{\frac{(d-1)(p-1)}{2}}
        = 0
    \end{align*}
    for all $0 < \lambda \le 1$.
    Therefore, $S_{\omega}(T_{\perp}^{\lambda} f)$ achieves its maximum at $\lambda = 1$.
    Since $p \ge 1 +8/(d-1)$, we have \eqref{eq:K(perp)} for all $\lambda \ge 1$.
    Hence, since $K_{\perp}(f) = 0$, $K(T_{\perp}^{\lambda} f) \to - \infty$ as $\lambda \to \infty$.
    On the other hand, $K(T_{\perp}^{\lambda} f) \to \mu \|\pt_1 f\|_{L^2(\R^d)}^2 + 2 \|\pt_1^2 f\|_{L^2(\R^d)}^2 > 0$ as $\lambda \to 0$.
    Therefore, there exists $\lambda_0 \in (0,\infty)$ such that $K(T_{\perp}^{\lambda_0} f) = 0$.
    From \eqref{eq:alter_vc} we have
    \begin{align*}
        S_{\omega}(Q_{\omega})
        \le S_{\omega}(T_{\perp}^{\lambda_0} f)
        \le S_{\omega}(T_{\perp}^{1} f)
        = S_{\omega}(f).
    \end{align*}
    This implies $S_{\omega}(Q_{\omega}) \le m_{\omega, \perp}$.
\end{proof}

\subsection*{Data availability}
There is no available data related to this work.

\subsection*{Conflict of Interest}
The authors declare that there are no conflicts of interest.

\subsection*{Acknowledgement}
The first author is supported by KAKENHI Grant-in-Aid for Early-Career Scientists No. JP24K16947 and partially by KAKENHI Grant-in-Aid for Scientific Research (B) No. JP26K00612.
The second author is supported by KAKENHI Grant-in-Aid for Early-Career Scientists No. JP23K13003 and partially by JST [Moonshot R\&D][Grant Number JPMJMS25A5].



\end{document}